\documentclass[12pt,a4paper]{amsart}
\usepackage[top=35mm, bottom=30mm, left=30mm, right=30mm]{geometry}
\usepackage{mathptmx}
\usepackage{mathrsfs}

\usepackage{verbatim}
\usepackage{url}
\usepackage[all]{xy}
\usepackage{xcolor}
\usepackage[colorlinks=true,citecolor=blue]{hyperref}

\usepackage{amsmath,amsthm}
\usepackage{amssymb}

\usepackage{enumerate}

\usepackage{graphicx}

\newtheorem{thm}{Theorem}[section]

\newtheorem{lem}[thm]{Lemma}

\newtheorem{ques}[thm]{Question}

\theoremstyle{definition}

\numberwithin{equation}{section}

\begin{document}
\title[A dynamical argument for a Ramsey property]{A dynamical argument for a Ramsey property}

\author[E.~Shi]{Enhui Shi}
\address{Soochow University, Suzhou, Jiangsu 215006, China}
\email{ehshi@suda.edu.cn}

\author[H.~Xu]{Hui Xu}
\address{Soochow University, Suzhou, Jiangsu 215006, China}
\email{20184007001@stu.suda.edu.cn}

\begin{abstract}
We show by a dynamical argument that there is a positive integer valued function $q$ defined on positive integer set $\mathbb N$ such that $q([\log n]+1)$ is a super-polynomial with respect to positive $n$ and
\[\limsup_{n\rightarrow\infty} r\left((2n+1)^2, q(n)\right)<\infty,\] where $r(\ ,\ )$ is the opposite-Ramsey number function.
\end{abstract}

\date{\today}

\setcounter{page}{1}

\maketitle

\section{Introduction and preliminaries}

For positive integers $p$ and $q$, we define the \textbf{opposite-Ramsey number} $r(p,q)$ to be the maximal number $k$ for which every edge-coloring of the complete graph $K_q$ with $p$ colors yields a monochromatic complete subgraph of order $k$ (the \textbf{order} of a graph means the number of its vertices).

The following is implied by the well-known Ramsey's theorem.

\begin{thm}
Let $p$ be a fixed positive integer. Then \[\liminf_{q\rightarrow\infty} r\left(p, q\right)=\infty.\]
\end{thm}

One may expect that if $p=p(n)$ and $q=q(n)$ are positive integer valued functions defined on  $\mathbb N$ and the speed of $q(n)$ tending to infinity is much faster than that of $p(n)$ as $n$ tends to infinity, then we still have \[\liminf_{n\rightarrow\infty} r\left(p(n), q(n)\right)=\infty.\]

The purpose of the paper is to show by a dynamical argument that this is not true in general even if $p(n)$ is a polynomial and $q([\log n]+1)$ is a super-polynomial.
By a \textbf{super-polynomial}, we mean a function $f: \mathbb{N}\rightarrow \mathbb{R}$ such that  for any polynomial $g(n)$, \[\liminf_{n\rightarrow\infty}\frac{|f(n)|}{|g(n)|}=\infty.\]

Let $(X,d)$ be a compact metric space. For any $\varepsilon>0$, let $N(\varepsilon)$ denote the minimal number of subsets of diameter at most $\varepsilon$ needed to cover $X$. The \textbf{lower box dimension} of $X$ is defined to be

\begin{equation}
\underline{\dim}_B(X,d)=\liminf_{\varepsilon\rightarrow 0} \frac{\log N(\varepsilon)}{\log1/\varepsilon}.
\end{equation}

For a subset $E$ of $X$ and $\varepsilon>0$, we say $E$ is $\varepsilon$-\textbf{separated} if for any distinct $x,y\in X$, $d(x,y)\geq \varepsilon$. Let $S(\varepsilon)$ denote the cardinality of a maximal $\varepsilon$-separated subset of $X$. It is easy to verify $N(\varepsilon)\leq S(\varepsilon)\leq N(\varepsilon/2)$. Thus

\begin{equation}
\underline{\dim}_B(X,d)=\liminf_{\varepsilon\rightarrow 0} \frac{\log S(\varepsilon)}{\log1/\varepsilon}.
\end{equation}

Furthermore, it is easy to see that

\begin{equation}
\underline{\dim}_B(X,d)=\liminf_{n\rightarrow \infty} \frac{\log S(1/n)}{\log n}.
\end{equation}

We use $\dim(X)$ to denote the topological dimension of $X$. It is well known that the topological dimension of $X$ is always no greater than its lower box dimension with respect to any compatible metric.

A continuous action $G\curvearrowright X$ of group $G$ on $X$ is said to be \textbf{expansive} if there exists $c>0$ such that for any two distinct points $x,y\in X$, $\sup_{g\in G}d(gx,gy)> c$. For $v=(v_1,\cdots,v_k)\in \mathbb{Z}^k$, let $|v|$ denote $\max\{|v_1|,\cdots,|v_k|\}$.

The following lemma is due to  T. Meyerovitch and M. Tsukamoto.

\medskip
\begin{lem}~\cite[Lemma 4.4]{MT}\label{compatible metric}
Let $k$ be a positive integer and $T:\mathbb{Z}^k\times X\rightarrow X$ be a continuous action of $\mathbb Z^k$ on a compact metric space $(X,d)$. If the action is expansive, then there exist $\alpha>1$ and a compatible metric $D$ on $X$ such that for any positive integer $n$ and any two distinct points $x,y\in X$ satisfying $D(x,y)\geq \alpha^{-n}$, we have
\[\max_{v\in\mathbb{Z}^k, |v|\leq n} D(T^{v}x,T^{v}y)\geq \frac{1}{4\alpha}.\]
\end{lem}


\begin{lem}\label{sup-poly}
If $(X,d)$ is a compact metric space of infinite dimension, then $S(1/n)$ is a super-polynomial with respect to variable $n$.
\end{lem}
\begin{proof}
Since $\dim(X)\leq \underline{\dim}_B(X,d)$, we have
\begin{equation}
\underline{\dim}_B(X,d)=\liminf_{n\rightarrow \infty} \frac{\log S(1/n)}{\log n}=\infty.
\end{equation}
Thus, for any positive integer $k$, $\liminf_{n\rightarrow \infty} \frac{S(1/n)}{n^k}=\infty$.
\end{proof}

\section{Main results}
For a positive real number $x$, we use $[x]$ to denote its integer part.
\begin{thm}\label{main thm}
There is a function $q:\mathbb{N}\rightarrow \mathbb{R}$ such that $q([\log n]+1)$ is a super-polynomial and
\[\limsup_{n\rightarrow\infty} r\left((2n+1)^2, q(n)\right)<\infty.\]
\end{thm}

\begin{proof}
Let $T: \mathbb{Z}^2\times X\rightarrow X$ be an expansive continuous action on a compact metric space $(X,d)$ of infinite dimension (see \cite{SZ} where an expansive $\mathbb{Z}^2$-action on $\mathbb{T}^\infty$ was constructed). By Lemma \ref{compatible metric}, there exist $\alpha>1$ and a compatible metric $D$ on $X$ such that for any positive integer $n$ and any two distinct points $x,y\in X$ with $D(x,y)\geq \alpha^{-n}$,
\[\max_{v\in\mathbb{Z}^2, |v|\leq n} D(T^{v}x,T^{v}y)\geq \frac{1}{4\alpha}.\]

For each $n\in\mathbb{N}$, let $V_n$ be a maximal $\alpha^{-n}$-separated set of $ (X, D)$. Hence $|V_n|=S(\alpha^{-n})$.
Let $G_n$ be the complete graph $K_{S(\alpha^{-n})}$ whose vertex set is $V_n$. Now we use the color set $C_n=\{v\in\mathbb{Z}^2: |v|\leq n\}$ to color the edges of $G_n$. Since $V_n$ is $\alpha^{-n}$-separated, for any two distinct points $x,y\in V_n$, $D(x,y)\geq \alpha^{-n}$. By Lemma \ref{compatible metric}, there exists $v\in C_n$ such that $D(T^{v}x,T^{v}y)\geq \frac{1}{4\alpha}$. Then we color the edge $\{x,y\}$ by $v$. By the definition of opposite-Ramsey number, there is a monochromatic complete subgraph $H_n$ of order $r\left((2n+1)^2, S(\alpha^{-n})\right)$.

 By Lemma \ref{sup-poly}, $S(1/n)$ is a super-polynomial. Let $q(n)=S(\alpha^{-n})$. Thus $q([\log n]+1)$ is a super-polynomial with respect to positive $n$. Assuming that the conclusion of the Theorem is false, we have
\[\limsup_{n\rightarrow\infty} r\left((2n+1)^2, q(n)\right)=\infty.\]
Therefore, there is an increasing subsequence $(n_i)$ of positive integers such the the sequence of orders of $H_{n_i}$ is unbounded. Since $H_{n_i}$ is monochromatic,  there exists $v_{n_i}\in C_{n_i}$ such that the image of vertex set of $H_{n_i}$ under $T^{v_{n_i}}$ is $\frac{1}{4\alpha}$-separated.  These imply that there are arbitrarily large $\frac{1}{4\alpha}$-separated subsets of $X$, which contradicts the compactness of $X$. Thus we complete the proof.
\end{proof}

\section{Comparison with Classical Ramsey number}
For any positive integers $k$ and $g$, the \textbf{Ramsey number} $R_{g}(k)$ is defined to be the minimal number $n$ for which every edge-coloring of the complete graph $K_n$ with $g$ colors yields a monochromatic complete subgraph of order $k$.

By Corollary 3 of Greenwood and Gleason in \cite{GG}, $R_{g}(k)$ has an upper bound $g^{gk}$. In \cite{LR} Lefmann and R\"{o}dl obtained a lower bound $2^{\Omega(gk)}$ for $R_{g}(k)$. Thus
\begin{equation}\label{bound for R}
2^{\Omega((2n+1)^2k)}\leq R_{(2n+1)^2}(k)\leq \left((2n+1)^2\right)^{(2n+1)^2k}.
\end{equation}
Suppose $r\left((2n+1)^2, q(n)\right)=r<\infty$. Then it implies that
\begin{itemize}
  \item [(1)] every edge-coloring of complete graph $K_{q(n)}$ with $(2n+1)^2$ colors yields a monochromatic complete subgraph of order $r$, hence
      \begin{equation}\label{eq of upper bound}
      q(n)\geq R_{(2n+1)^2}(r);
      \end{equation}
  \item [(2)] there exists an edge-coloring of $K_{q(n)}$ with $(2n+1)^2$ colors such that there is no monochromatic complete subgraph of order $r+1$, hence
      \begin{equation}
      q(n)\leq R_{(2n+1)^2}(r+1).
      \end{equation}
\end{itemize}

Thus $q(n)$ gives a lower bound of $R_{(2n+1)^2}(r+1)$ and an upper bound of $R_{(2n+1)^2}(r)$. By Theorem \ref{main thm}, every expansive $\mathbb{Z}^2$-action on a compact metric space of infinite dimension gives rise to such a $q(n)$. In addition, there is a positive integer $r$ and an increasing subsequence $(n_i)$ of positive integers such that for any $i\in\mathbb{N}$, $r\left((2n_i+1)^2, q(n_i)\right)=r$. Therefore, we obtain a lower bound of $R_{(2n_i+1)^2}(r+1)$ and an upper bound of $R_{(2n_i+1)^2}(r)$ for each $i\in\mathbb{N}$.

If $q(\log n)$ is a super-polynomial, then we claim that for any $A\geq 0$,
\begin{equation}\label{eq of q}
\liminf_{n\rightarrow\infty} \frac{q(n)}{A^n}=\infty.
\end{equation}
In fact, take a positive integer $m$ such that $e^m\geq A$. Then
\begin{eqnarray*}
\liminf_{n\rightarrow\infty} \frac{q(n)}{A^n}\geq \liminf_{n\rightarrow\infty} \frac{q(n)}{e^{mn}}
=\liminf_{n\rightarrow\infty} \frac{q(\log (e^n))}{(e^n)^m}=\infty.
\end{eqnarray*}

The lower bound of $R_{(2n+1)^2}(r+1)$ obtained by (\ref{bound for R}) is $2^{\Omega\left((2n+1)^2(r+1)\right)}$ which is also faster than any exponential growth. If there is an expansive $\mathbb{Z}$-action on a compact metric space of infinite dimension, then we can get a lower bound for $R_{2n+1}(r+1)$ which is faster than the classical bound $2^{\Omega\left((2n+1)(r+1)\right)}$. Unfortunately, in \cite{M} Ma\~{n}\'{e} showed that such action does not exist. If we can construct an expansive $\mathbb{Z}^2$-action on a compact metric space of infinite dimension such that the condition $D(x,y)\geq \alpha^{-n}$ in  Lemma \ref{compatible metric} can be replaced by $D(x,y)\geq \alpha^{-n^2}$, then we can show that $q(n)$ obtained in Theorem \ref{main thm} satisfies that $q([\sqrt{\log n}]+1)$ is a super-polynomial. Then it satisfies $\liminf_{n\rightarrow\infty} \frac{q(n)}{A^{n^2}}=\infty$. Hence $q(n)$ is faster than the classical lower bound $2^{\Omega\left((2n+1)^2(r+1)\right)}$. Therefore, we leave the following question.
\begin{ques}\label{question}
Is there an expansive $\mathbb{Z}^2$-action on a compact metric space $(X,d)$ of infinite dimension and $\alpha>1$ such that for any positive integer $n$ and any two distinct points $x,y\in X$ satisfying $d(x,y)\geq \alpha^{-n^2}$, we have
\[\max_{v\in\mathbb{Z}^2, |v|\leq n} d(T^{v}x,T^{v}y)\geq \frac{1}{4\alpha}.\]
\end{ques}
A positive answer to Question \ref{question} can give a better estimate of the lower bound of $R_{(2n_i+1)^2}(r+1)$, where $(n_i)$ and $r$ come from the system. By (\ref{eq of upper bound}), a negative answer also gives a better estimate of the upper bound of $R_{(2n_i+1)^2}(r)$.

Finally, we remark that the above comparison between $q(n)$ and the bound of Ramsey is only for a subsequence of positive integers. However, dealing with a concrete system we may obtain more information and can get a special edge-coloring of $K_{q(n)}$ as the proof of Theorem \ref{main thm}. Our method may give a new direction to estimate the bounds of Ramsey numbers and construct edge-colorings of big graphs.

\subsection*{Acknowledgements}
We would like to thank Professor Bingbing Liang and Professor Xin Wang for their helpful comments.


\begin{thebibliography}{999}

\bibitem{GG}  R. E. Greenwood and  A. M. Gleason, Combinatiorial Relations and Chromatic Graphs. \textit{Canadian Journal of Mathematics} 7(1955):1-7.

\bibitem{LR} H. Lefmann and V. R\"{o}dl, On canonical Ramsey numbers for complete graphs versus paths. \textit{J. Combin. Theory Ser. B} 58 (1993), no. 1, 1-13.


\bibitem{M} R. Ma\~{n}\'{e}, Expansive homeomorphisms and topological dimension, \textit{Trans. Amer. Math. Soc}. 252(1979), 313-319.

\bibitem{MT} T. Meyerovitch and M. Tsukamoto, Expansive multiparameter actions and mean dimension. \textit{Trans. Amer. Math. Soc.}  371 (2019), no. 10, 7275-7299.

\bibitem{SZ} E. Shi and L. Zhou, The nonexistence of expansive $\mathbb{Z}^d $ actions on graphs. \textit{Acta Math. Sin. (Engl. Ser.)} (2005) 1509-1514.
\end{thebibliography}
\end{document}